\documentclass[11pt,a4paper]{article}

\usepackage[OT1]{fontenc}
\usepackage[english]{babel}
\usepackage[utf8x]{inputenc}

\usepackage[warn]{textcomp}

\usepackage{amsmath,amssymb,amsfonts,mathrsfs}
\usepackage[amsmath,thmmarks]{ntheorem}

\usepackage{graphicx}
\usepackage{nicefrac}

\usepackage{tikz}
\usetikzlibrary{matrix}

\allowdisplaybreaks

\usetikzlibrary{arrows,positioning} 
\tikzset{
    >=stealth',
    punkt/.style={
           rectangle,
           rounded corners,
           draw=black, very thick,
           text width=6.5em,
           minimum height=2em,
           text centered},
    pil/.style={->,shorten <=0pt,shorten >=6pt,}
}

\usepackage[numbers,square]{natbib}

\usepackage{chngcntr}
\counterwithout{equation}{section}

\theoremprework{\begin{minipage}{\textwidth}}
	\theorempostwork{\end{minipage}}

\newtheorem{theoremw}{Theorem}[section]

\newtheorem{definition}[theoremw]{Definition}

\newtheorem{lemma}[theoremw]{Lemma}

\newtheorem{proposition}[theoremw]{Proposition}

\newtheorem{def+prop}[theoremw]{Definition+Proposition}

\newtheorem{theorem}[theoremw]{Theorem}

\theorembodyfont{\normalfont}

\theoremstyle{nonumberplain}
\theorembodyfont{\normalfont}
\theoremsymbol{\ensuremath{\square}}
\newtheorem{proof}{Proof}

\newcommand{\N}{\mathbb{N}}

\newcommand{\R}{\mathbb{R}}
\newcommand{\Z}{\mathbb{Z}}

\newcommand{\F}{\mathcal{F}}

\newcommand{\Prob}{\mathbb{P}}
\newcommand{\Ot}{\tilde{\mathbb{Q}}}
\newcommand{\Op}{\mathbb{Q}}
\newcommand{\Mt}{\mathcal{M}^t_1}
\newcommand{\Me}{\mathcal{M}_1}
\newcommand{\om}{\omega}

\newcommand{\h}{\mathsf{h}}

\newcommand{\Er}{\mathsf{E}}
\newcommand{\Pa}{\mathsf{P}}
\newcommand{\E}{\mathbb{E}}

\newcommand{\ind}{\mathbf{1}}

\title{The relation between quenched and annealed Lyapunov exponents in random potential on trees}
\author{Gundelinde Maria Wiegel}

\newcommand*{\SwitchToOpenAny}{\csname @openrightfalse\endcsname}
\newcommand*{\SwitchToOpenRight}{\csname @openrighttrue\endcsname}

\begin{document}

\maketitle

\begin{abstract}
In the first part of the article our subject of interest is a simple symmetric random walk on the integers which faces a random risk to be killed. This risk is described by random potentials, which in turn are defined by a sequence of independent and identically distributed non-negative random variables.
To determine the risk of taking a walk in these potentials we consider the decay of the Green function. 
There are two possible tools to describe this decay: The \textit{quenched Lyapunov exponents} and the \textit{annealed Lyapunov exponents}.
It turns out that on the integers we can state a precise relation between these two. The main tool for proving this is due to a special path property on the integers: if the random walk travels from one point to another one in finite time all the points between have to be passed in finite time as well. This allows using Ergodic Theorems. 
In the second part we show that the relation is also true in the case of Lyapunov exponents for simple symmetric random walks on $d$-regular trees and for random walks with drift.

{\let\thefootnote\relax\footnote
{2010 \emph{Mathematics Subject Classification.} 
60K37;  
05C05. 

\emph{Key words and phrases.} Random walks, random potential, Lyapunov exponents, homogeneous trees, relative entropy.

Supported by the Austrian Science Fund (FWF): W1230.

\textit{Graz University of Technology}. Email: wiegel@math.tugraz.at.
}}
\end{abstract}

\section{Introduction}
\subsection{Random walks with random killing}
We consider a \textit{simple random walk}  $(S_n)_{n\geq 0}$ on the integers $\Z$ with starting point $x$. At each point of time it jumps independently of all the steps before with probability $\frac{1}{2}$ to the right or to the left. The path measure of the random walk will be denoted by $\Pa_x$ and the expectation value with respect to $\Pa_x$ by $\Er_x$. The simple symmetric random walk is a Markov process. Moreover, it is spatially homogeneous. 

Furthermore, we attach to each site $x \in \Z$ a so called \textit{random potential} $\omega(x)$ which influences the movement of the random walk. We assume that $\omega:=(\omega(x))_{x\in \Z}$ is a sequence of nonnegative random variables which are independently and identically distributed (i.i.d.) by the common measure $\nu$ on $[0,\infty)$. From this we obtain the canonical probability space $(\Omega, \F, \Prob)$ described by
\[\Omega := [0,\infty)^{\Z}\]
with its usual Borelian product $\sigma$-algebra $\F$ and the product measure
\begin{align*}
\Prob:=\bigotimes_{x\in\Z}\nu \, .
\end{align*}
The expectation value derived with respect to $\Prob$ will be denoted by $\E$. We assume that the potentials are not concentrated at 0, that is $\nu\neq\delta_0$, to avoid the trivial case. Considering the standard shift $T_i:\Omega\rightarrow\Omega$, $(\omega(x))_{x\in \Z}\mapsto((\omega(x-i)_{x\in \Z})$ for $i\in \Z$ it is clear that $\Prob$ is shift invariant.   

The random potentials represent a certain risk of dying for the random walk. For a fixed realization of the environment $\omega=(\omega(x))_{x\in \Z}$ the random walk dies with probability 
\[q(x):=1- \exp(-\omega(x))\]
at each site $x \in \Z$ it reaches. If the random walk survives the site, it will uniformly choose the next site of its journey. Thus, the random walk will jump with probability $\nicefrac{(1-q)}{2}$ to the right or to the left, see Figure \ref{fig:label4}. So, given the potentials, it becomes dangerous to take a walk in these environments. 

\begin{figure}[h]
\centering
\begin{tikzpicture}[x=1cm,y=1cm]
\draw[help lines,step=1, thin, draw=white] (-4,-1) grid (3,4);

\foreach \Point in {(0,0),(2,0),(-2,0)}{
    \node at \Point {{\textbullet}};
    }
\draw[draw=black] (0,0) -- (3.5,0);
\draw[draw=black] (-3.5,0) -- (0,0);
   
\node[] at (0,0) {}
 edge[pil,thin,->, bend left=45] node[auto] {$\nicefrac{(1-q)}{2}$} (2,0)
 edge[pil,thin,->, bend left=30] node[auto] {\small{$q$}} (1.5,2.5);
\node[] at (-2,0) {}
 edge[pil,thin,<-, bend left=45,pos=0.4] node[auto] {$\nicefrac{(1-q)}{2}$} (0,0); 
\node[] at (1.5,2.5)(death) {\Large{\textbf{\textdied}}}; 
\node[] at (3.5,-0.3) {$\Z$};
\end{tikzpicture}
\caption{Simple random walk with random killing where $p(x):=1- \exp(-\omega(x))$ for a fixed realization of the potentials $\omega$.}
 \label{fig:label4}
\end{figure}
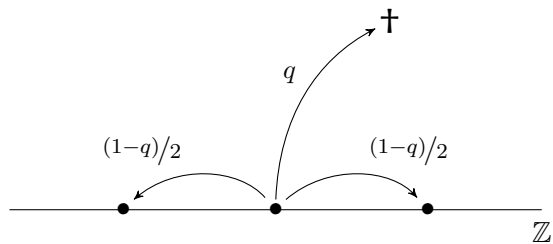

One measurement for the risk is the \textit{Green function}. For $x,y\in \Z$ and a realization of the potential $\omega$ we define it for the random walk with random killing by 
\begin{align*}
g(x,y,\omega):=\sum_{m>0}\Er_x\left[\exp\left(-\sum_{i=0}^m \omega(S_i) \right) \ind_{\{S_m=y\}}\right] \, .
\end{align*}
This is the expected number of visits in $y$ before the random walk starting at $x$ dies. We want to study the decay of $g$. In other words: how risky is it to walk around in this environment over long distances, that is if $|x-y|\rightarrow \infty$? Moreover $g$ is still a random variable in $\omega$ and we may ask the same for the averaged Green function with respect to $\Prob$. 

\subsection{Quenched and annealed Lyapunov exponents}
The Lyapunov exponents give us a precise description of the decay of Green's function. We start by shortly introducing them.
Therefore we are in need of another two-point function, which is closely related to $g$. 

First of all we define the stopping times $\tau_y,y \in \Z,$ given by
\[\tau_y:=\inf\{n\geq 0\, : \, S_n=y\}\]
for each $y \in \Z$. This is the first point of time where the original random walk $(S_n)_{n\geq 0}$ hits $y\in \Z$.

\begin{definition} \label{quencheddef}
For any realization of the potentials $\omega \in \Omega$ and $x,y \in \Z$ we define the two-point-functions
\begin{align*}
e(x,y,\omega):=&\Er_x \left[\exp \left(-\sum_{k=0}^{\tau_y-1} \omega (S_k) \right),\, \tau_y< \infty\right] \, ,\\
a(x,y,\omega):=& -\ln e(x,y,\omega) \, . 
\end{align*}
\end{definition}
The quantity $e(x,y,\omega)$ represents the probability of the random walk reaching $y$ after it started at $x$ and before it dies due to a fixed potential $\omega$. As before, $e(x,y,\cdot)$ is a random variable in $\omega$. Consequently, the expected probability of surviving a journey from $x$ to $y$ will be of interest to us hereafter. This procedure of taking the average is also called annealing the environment.
\begin{definition} \label{annealeddef}
Let $x,y \in \Z$. We define
\begin{align*}
f(x,y):=&\E[e(x,y,\omega)]\, ,\\
b(x,y):=&-\ln f(x,y)  \, .
\end{align*}
\end{definition}

To answer the above stated question we observe the behaviour of $e$ and $f$ in the long-run, that is when the distance of $x$ and $y$ tends to infinity. Looking for a precise description we turn to \textit{Lyapunov exponents}. There are two ways of dealing with the random potential – the \textit{quenched} and the \textit{annealed} case.

In the quenched case we look at the exponential decay of the survival rate for a frozen realization of the potentials. 
\begin{proposition} \label{existencequenched}
We suppose that $\nu$ has finite expectation. Then, for all $x\in \Z$ there exists  the limit
\[\alpha(x):= \lim_{n\rightarrow \infty} \frac{1}{n}a(0,nx,\omega)= \lim_{n\rightarrow \infty} \frac{1}{n}\E[a(0,nx,\omega)]=\inf_{n\in N} \frac{1}{n}\E[a(0,nx,\omega)]\,\ \]
$\Prob$-almost surely and in $L^1$. Moreover, it holds that
\begin{align}
\alpha(x)=\E[a(0,x,\omega)]=|x|\, \E[a(0,1,\omega)]\, \label{specialformula}
\end{align}
so that  $\alpha(x)$ is a non-random norm.
\end{proposition}
The limit $\alpha(x)$ for $x \in \Z$ is called \textit{quenched Lyapunov exponent}. The proof of the existence we find in \cite[Proposition 4]{zerner}. In this paper Zerner introduced the quenched Lyapunov exponents for simple symmetric random walks on $\Z^d$, $d \geq 1$. Just representation (\ref{specialformula}) is special on the integers, see \cite[Proposition 10]{zerner}. On the integers the function $a$ is not only subadditive, but also additive. 
This property has its origin in the path properties of the simple symmetric random walk on the integers. The random walk can step just one site to the right or to the left at once (without involving the random potentials). That is, it cannot jump across a site, which implies 
\begin{align}
\{\tau_z < \infty \} \subseteq \{\tau_y < \infty\}\, \label{stoppingtimeinclusion}
\end{align} 
for all sites $y$ between $x$ and $z$. As this property is essential to the upcoming part we sketch the proof according to \cite[Proposition 10]{zerner}.
\begin{lemma} \label{additivitylemma}
Let $\omega$ be a realization of the potential and $x,y,z \in \Z$, with $y$ between $x$ and $z$. Then $a(x,z,\omega) = a(x,y,\omega) + a(y,z,\omega)$.
\end{lemma}
\begin{proof}
Let $\omega \in \Omega$ be a realization and without loss of generality $x \leq y \leq z \in \Z$. By (\ref{stoppingtimeinclusion}) we can modify function $e$ in the following way:
\begin{align*}
e(x,z,\omega)&=\Er_x \left[ \exp\left(-\sum_{i=0}^{\tau_z-1} \omega(S_i)\right) \cdot \ind_{\{\tau_z < \infty\}}\right]\\[2mm]
&=\Er_x \left[ \exp\left(-\sum_{n=0}^{\tau_z-1} \omega(S_i)\right) \cdot \ind_{\{\tau_z < \infty\}} \cdot \ind_{\{\tau_y < \infty\}} \right] \, .
\end{align*} 
This, together with the tower property and the strong Markov property for the stopping time $\tau_y$ yields
\begin{align}
e(x,z,\omega)&= e(x,y,\omega) \, e(y,z,\omega) \,  \label{multiplicationproperty}
\end{align} 
which implies the additivity of $a$. 
\end{proof}

The additivity of $a$ allows to use Birkhoff's and Neumann's ergodic theorems in the proof of Propostion \ref{existencequenched} instead of the subadditive limit theorem in higher dimensions. This provides formula (\ref{specialformula}) and proves Proposition \ref{existencequenched}. 

In a so called shape-theorem \cite[Theorem 8]{zerner} Zerner proved that the quenched Lyapunov exponents describe the exponential decay of the Green function as follows:
\begin{align*}
\lim_{|x|\rightarrow \infty} \frac{-\ln g(0,x,\omega)}{\alpha(x)}=\lim_{|x|\rightarrow \infty} \frac{-\ln e(0,x,\omega)}{\alpha(x)}=1
\end{align*}
for $\Prob$ almost all $\omega \in \Omega$ and in $L^1(\Prob)$. This is just the basic result concerning the decay of $g$, respectively $e$. For further reading we refer again to \cite{zerner} and to \cite{mourrat} and \cite{MountfordKosZer}.\\

In the case of the annealed Lyapunov exponents we have similar results. Here the expectation of the survival rate with respect to the potentials is taken before looking at its decay in the long run. 
\begin{proposition} \label{existenceannealed}
We suppose that $\nu$ has finite expectation. Then for all $x \in \Z$ there exists the limit
\[\beta(x):= \lim_{n\rightarrow \infty} \frac{1}{n}b(0,nx)=\inf_{n\in N} \frac{1}{n} b(0,nx)\,\ \] 
and $\beta$ is a norm. 
\end{proposition}
We call $\beta(x)$ for each $x\in \Z$ the \textit{annealed Lyapunov exponent}. Flury has proven the existence of $\beta$ also on higher dimensional lattices in \cite{flury}. We find there analogous shape theorems for the averaged Green function, for example
\begin{align*}
\lim_{|x|\rightarrow \infty} \frac{-\ln \E[g(0,x,\omega)]}{\beta(x)}=1 \, .
\end{align*}
Indeed, in \cite{flury} a slightly more general case is treated, allowing other influences of the random potential on the random walk besides the extinction probability $1-\exp(-\omega(x))$ at each site $x\in \Z$.

\subsection{Results}
The main result of this paper is Theorem \ref{theorem}. It is a precise formula in terms of entropy which connects the annealed Lyapunov exponents with the quenched ones. This is the discrete analogon of Theorem 1.9 in \cite{sznitman}. The main tool for proving this is the additivity of $a$, see Lemma \ref{additivitylemma}. In section \ref{trees} we will introduce Lyapunov exponents on infinite regular trees. We will see that there – as a corollary of the aforegoing Section – the relation between annealed and quenched Lyapunov exponents holds as well. Moreover, we can generalize the result for non-symmetric simple random walks. \\

In addition to the aforementioned literature we recommend  \cite{sznitman98} for a general overview and the background on Lyapunov exponents to the reader. There, mainly Lyapunov exponents for Brownian Motion moving in Poissonian Potential are treated. Some further current results related to Lyapunov exponents appear in \cite{YilmazFree} and \cite{RassoulFreeEnergy}. Instead of Lyapunov exponents the authors observe quenched free energy and quenched point-to-point free energy in a more general situation of the random environment. More precisely it is a generalization of random walks in random potential and random walk in random environment. As one result they develop a variational formula for the quenched free energy also using entropy.  

Furthermore there are lots of related results for random walks in random environment. Zerner provides an introduction to Lyapunov exponents for RWRE in \cite{zernerPHD2RWRE}. In \citep{gantertRWRE1} random walks in random environment especially on the integers are discussed and they provide a relation of some rate functions for the quenched to the annealed random environment.    

\section{The relation between quenched and annealed Lyapunov exponents on the integers}\label{relation}
We have already seen that the quenched and the annealed Lyapunov exponents differ in their treatment of the random potential. Using the quenched approach, we observe the exponential decay for a typical realization of the random potentials. While using the annealed approach, we look at the environment as averaged. We now aim to describe this difference in greater detail. By applying Jensen's inequality it is easy to conclude that for the quenched and annealed Lyapunov exponents it holds
\[\alpha(x)\geq \beta(x)\, \]
for each $x\in \Z$. This relation also holds for simple symmetric random walks on $\Z^d$ with $d\geq 2$. But it turns out that in the case of a random walk on the integers we can prove an explicit formula for this relation. Before stating the main result we recall some definitions concerning entropy.

We consider the canonical projection $\pi_I: \Omega \rightarrow \Omega^I$
where $I$ is a finite subset of $\Z$ and $\Omega^I:=[0,\infty)^I$. The corresponding product $\sigma$-algebra is denoted by $\F_I$.
For any  probability measure $\Op$ on $\Omega$ we consider its restriction $\mathbb{Q}^I$ to $\F_I$. Now let $\Ot$ be a second probability measure on $\Omega$, which is absolutely continuous with respect to $\Op$ on $\F_I$ for a given finite intervall $I \subset \Z$. Hence we can define the Radon-Nikodym derivative $f_I$ on $\F_I$ which is a positive and $\F_I$-measurable function. 

\begin{definition}[Relative Entropy] \label{entropydef}
Let $\mathbb{Q}$ and $\Ot$ be two probability measures on $\Omega$ and $I \subset \Z$ a finite intervall. Then we call
\[H_I(\Ot|\Op):= 
\begin{cases}
\int_{\Omega} f_I \ln f_I \, d\Op  & \mbox{if}\,\ \Ot\ll \Op \,\ \text{on} \,\ \F_I  \\
\infty & \mbox{else}
\end{cases} \]
the \textbf{relative entropy of $\Ot^I$ with respect to $\Op^I$}.
\end{definition}
In most of the literature on Information Theory,
the relative entropy is denoted $D(\Ot\| \Op)$
and also known as the Kullback-Leibler divergence.
Equivalently, if $\Ot\ll \Op$ on $\F_I$  the relative entropy can be expressed by
$H_I(\Ot|\Op)=\int_{\Omega} \ln f_I \, d\Ot$. As the function $x \ln x$ is strictly convex we see by Jensen's inequality that $H_I$ is nonnegative and zero if and only if the two measures coincide. If $\Ot$ and $\Op$ are shift invariant for $(T_i)_{i \in \Z}$, the relative entropy $H_I(\Ot|\Op)$ is shift invariant as well. We denote by $ \Mt(\Omega)$ all shift invariant probability measures on $\Omega$. Furthermore, under the additional assumption that $\Op$ is a product measure on $\Omega$, the relative entropy becomes strongly superadditive,
 see \cite[Proposition 15.10]{georgii88}. Thus under these assumptions the  subadditive limit theorem guarantees the existence of 
\begin{align}
H(\Ot|\Op):=\lim\limits_{n \rightarrow \infty} \frac{1}{|I_n|} H_{I_n}(\Ot|\Op) = \sup_{I_n} \frac{1}{|I_n|} H_{I_n}(\Ot|\Op) 
\end{align} 
with $(I_n)_{n\in \N}$ a sequence of intervals  satisfying $I_n \subseteq I_{n+1}$ for each $n\in\N$ and $|I_n|{\rightarrow} \infty$ when $n\rightarrow \infty$. We call $H(\Ot|\Op)$ the \textit{specific relative entropy of $\Ot$ with respect to $\Op$}.
The distribution of the random potentials $\Prob$ fulfills all the above assumptions and the specific relative entropy $H(\Op|\Prob)$ is well-defined for any shift invariant probability measure $\Op$ on $\Omega$. 

Additionally, in relation with the quenched Lyapunov exponent, we define for $r< 0$ the functions $F_r: \Omega \rightarrow [0, \infty]$ by
\begin{align*}
F_r(\omega)&:=-\ln \Er_0\left[\exp\left(-\sum_{i=0}^{\tau_1-1} \omega(S_i)\right), \tau_1 < \tau_r , \tau_1 < \infty \right] 
\end{align*}
and their counterpart $F$ by omitting the inequality $\tau_1 < \tau_r$. It is easy to see that $F(\omega)=-\ln e(0,1,\omega)=a(0,1,\omega)$ and $0\leq F \leq \om(0)+\ln 2$. Consequently we have 
\begin{align}
\alpha(1)=\E[F(\omega)] \, .\label{comparison}
\end{align} 

Keeping in mind that $H(\Prob|\Prob)=0$ and formula (\ref{comparison}), the description of the relation between annealed and quenched Lyapunov exponents is the following:

\begin{theorem} \label{theorem} Let $(\Omega, \F, \Prob)$ be the probability space defined above where $\Prob=\otimes_{x\in\Z}\nu$. Moreover, we assume that $\nu$ has finite expectation. Then
\begin{align*}
\beta(1)= \inf_{\Op} \{\E^{\Op}[F(\omega)] + H(\Op|\Prob)\} 
\end{align*}
and the infimum runs over all shift invariant probability measures $\Op$ on $\Omega$.
\end{theorem}

A similar relation for Brownian motion moving in Poissonian potential was proven in \cite[Theorem 1.9]{sznitman}. We will follow Sznitman's ideas and split our proof in two parts proving the upper and lower bound 
\begin{align}
\beta(1)&\leq \inf_{\Op \in \Mt(\Omega)} \{E^{\Op}[F(\omega)] + H(\Op|\Prob)\} \label{firstpart} \\
\beta(1)&\geq \inf_{\Op \in \Mt(\Omega)} \{E^{\Op}[F(\omega)] + H(\Op|\Prob)\}  \label{secondpart}
\end{align} 
separately as well. 

\subsection{Proof of the upper bound}
\begin{proof} \textit{Upper bound} As a consequence of the multiplication property (\ref{multiplicationproperty}) of $e$ and the spatial homogeneousity we connect the definition of $f$ with $F$. The latter follows from the observation that $\{\tau_1 < \tau_r,\tau_1 < \infty\} \subseteq \{\tau_1 < \infty\}$.
\begin{align}
 f(0,n+1) &= \E\left[ \prod_{k=0}^{n} e(k,k+1,\omega)\right] = \E\left[ \prod_{k=0}^{n} e\left(0,1,T_{-k}(\omega)\right)\right]\notag\\
&=\E\left[\exp \left(\sum_{k=0}^{n} -F \circ T_{-k}(\omega)\right)\right] \label{frepresentation} \\
&\geq \E\left[\exp \left(\sum_{k=0}^{n} -F_r \circ T_{-k}(\omega)\right)\right] \, .  \label{firstest}
 \end{align}
Now let us consider an arbitrary shift invariant probability measure $\Op \in \Mt(\Omega)$. Thus $H(\Op|\Prob)$ is well defined and we assume additionally  
\[H(\Op|\Prob) < \infty \,\ \text{and} \,\
\E^{\Op}[F] < \infty \, .\]
This implies that $\Op \ll \Prob$ on $\F_{I}$ for each finite intervall $I \subset \Z$ and we define the corresponding Radon-Nikodym derivatives $f_I$ on $\F_I$, $I \subset \Z$. It is easy to see that each $f_I$ is strictly positive $\Op$-almost surely and for each finite $I \subset \Z$ it holds that
\begin{align}
\int_{\Omega} g\, d\Prob = \int_{\Omega} \frac{g}{f_I}\, d\Op \label{radonumschreiben}
\end{align}
for any $\F_I$ measurable function $g:\Omega \rightarrow \R$. The movement of the random walk in (\ref{firstest}) is restricted by the definition of $F_r$ to $A:=[r+1,n+1]$ and the function $\exp \left(\sum_{k=0}^{n} -F_r \circ T_{-k}(\omega)\right)$ is $\F_A$-measurable.
Then, by (\ref{firstest}),(\ref{radonumschreiben}) and Jensen's inequality we obtain:
\begin{align*}
f(0,n+1) &\geq \E\left[\exp \left(\sum_{k=0}^{n} -F_r \circ T_{-k}(\omega)\right)\right] \\
&= \E^{\Op}\left[\exp \left(\sum_{k=0}^{n} -F_r \circ T_{-k}(\omega)\right)\cdot \frac{1}{f_A}\right]\\
&\geq  \exp \left( \E^{\Op}\left[\sum_{k=0}^{n} -F_r \circ T_{-k} - \ln f_A \right] \right) \\
&= \exp \left( \E^{\Op}\left[\sum_{k=0}^{n} -F_r \circ T_{-k}\right]  - H_A(\Op|\Prob) \right) \, .
\end{align*}
Taking the $n+1$-th root, the negative logarithm and the limit $n$ to infinity on both sides yields
\begin{align*}
\beta(1) &\leq \E^{\Op} [F_r]  +  H(\Op|\Prob) \, . 
\end{align*}
Obviously, $\ind_{H(1)<H(r)}$ converges monotonously from below to $\ind$  for $r \rightarrow - \infty$ and we can replace $F_r$ by $F$.
Because of the finite expectation of $\nu$ there is at least $\Prob$ such that the right side of the last inequality is finite and we may deduce
\[\beta(1)\leq \inf_{\Op \in \Mt(\Omega)} \{E^{\Op}[F(\omega)] + H(\Op|\Prob)\}\,. \]
\end{proof}

\subsection{Proof of the lower bound}
In order to prove the lower bound (\ref{secondpart}) we need to make use of some statements from the theory of large deviations, more specifically from process level large deviations theory. 
A sequence  $\{\mu_n\}_{n\in \N}$ of probability measures on a polish space $E$ satisfies a \textit{large deviation principle with rate function $J$ and normalization $r_n$} if the following two inequalities hold: 
\begin{align}
\limsup_{n\rightarrow \infty} \frac{1}{r_n} \mu_n(F) &\leq -\inf_{x\in F} J(x) & \,\ &\forall F \subset E \, \text{closed} \label{ldpup} \\
\liminf_{n\rightarrow \infty} \frac{1}{r_n} \mu_n(G) &\geq -\inf_{x\in G} J(x) & \,\ &\forall G \subset E \, \text{open} \,
\end{align}
 where $J: E \rightarrow [0,\infty]$ is a lower semi-continuous function and $\{r_n\}_{n\in \N} \subset \R_{+}$ is a sequence of positive real numbers with $r_n \nearrow \infty$.
We abreviate by writing that $\text{LDP}(\mu_n,r_n,J)$ holds if all these requirements are satisfied. When the sets $\{x \in E: J(x) \leq c\}$ are compact for all $c \in [0;\infty)$ the rate function $J$ is said to be \textit{good}. 
Process level large deviations theory is concerned with the asymptotics of the distributions of the empirical measures of a whole process. 
Therefore we consider the random potentials $(\omega(i))_{i\in \Z}$ as a process with state space $[0,\infty)$.
For this process the $n$-th empirical measure $R_n : \Omega \rightarrow \mathcal{M}_1(\Omega)$ is defined by 
\begin{align*} 
R_n(\omega) &:= \frac{1}{|I_n|} \sum_{i \in I_n} \delta_{T_{-i}(\omega)} 
\end{align*}
with the normalizing sequence of intervalls $(I_n)_{n\in \N}$ in $\Z$ given by $I_n:=\{i \in \Z : -n < i < n\}$ for each $n \in \N$.
Every probability measure $\Op$ on $\Omega$ induces a distribution $\mu_n \in \Me(\Me(\Omega))$ of the empirical measure $R_n$. The distributions of the empirical measures $(R_n)_{n\in \N}$ satisfy the LDP with good rate function $J:\Me(\Omega) \rightarrow [0,\infty]$ defined by
\begin{align}
J(\Op):= \begin{cases} H(\Op|\Prob) &\text{if} \,\ \Op \in \Mt(\Omega) \\
\infty &\text{else} \,  \label{ratefunction}
\end{cases}
\end{align}
and the normalizing sequence $(|I_n|)_{n\in \N}$, see \cite[Theorem 6.13]{rassoul}.

Basically the following proof is an application of a version of Varadhan's theorem.

\begin{proof} \textit{Lower bound} 
Let  $\Phi: \mathcal{M}_1(\Omega) \rightarrow [-\infty,0]$ be the function defined by
\[\Phi(\Op):=\E^\Op[-F]\, .\] 
We have already seen that $F$ is a positive function. Hence $-F$ is bounded from above. Moreover, $F$ is a continuous function with respect to the product topology on $\Omega$. This is a consequence of the continuity of parameter dependent integrals. Recall that the weak convergence topology on $\Me(\Omega)$ is the coarsest topology such that for each bounded and continuous $f\in C_b(\Omega)$ the map $\Me(\Omega) \rightarrow \R$, $\rho \mapsto \int_{\Omega} f d\rho$ is continuous. 
Let $\rho \in \Me(\Omega)$ and let $(\rho_n)_{n\in \N} \in \Me(\Omega)$ be a sequence in $\Me(\Omega)$ which converges weakly to $\rho$. Then it holds that  
\[\limsup_{n \rightarrow \infty} \E^{\rho_n}[-F] \leq \E^{\rho}[-F] \, \]
as $-F$  is continuous and bounded from above. Consequently $\Phi$ is an upper-semi-continuous function with respect to the weak convergence topology. 

By (\ref{ratefunction}) we know that $LDP(\mu_n,|I_n|,J)$ holds with rate function $J$ and as $-F$ is negative, the set ${\{\Op\in \Me(\Omega): \Phi(\Op)\geq L\}}$ is empty for $L > 0$.

Thus applying \cite[Lemma 2.1.8]{deuschel} we conclude 
\begin{align*}
\limsup_{n\rightarrow\infty} \frac{1}{|I_n|} \ln \int\limits_{\Me(\Omega)} \exp \big[|I_n| \Phi(\Op)\big] \, d\mu_n(\Op) \notag \\
 \leq \sup_{\{\Op\in \Me(\Omega)\}} \big[\Phi(\Op)-J(\Op)\big] \,  
\end{align*}
and by some transformations and the shift invariance of $\Prob$ we obtain
\begin{align*}
\int\limits_{\Me(\Omega)} \exp \big(|I_n| \Phi(\Op)\big) \, d\mu_n(\Op)&=\E^{\Prob} \left[ \exp \big((2n-1) \E^{R_n}[-F] \big)\right] \\[-5mm]
&=\E^{\Prob} \left[ \exp \left(- \sum_{i =-(n-1)}^{n-1} F \circ T_{-i}(\omega)\right) \right] \\
& =\E^{\Prob} \left[ \exp \left(- \sum_{i=0}^{2n-1} F \circ T_{-i}(\omega)\right) \right] \, .
\end{align*}
Moreover, by (\ref{frepresentation}) it holds that
\begin{align*}
\beta(1) \geq \liminf_{n\rightarrow\infty} -\frac{1}{2n-1} \ln \E\left[ \exp \left(- \sum_{k=0}^{2n-1} F \circ T_{-k}(\omega)\right) \right]
\end{align*}
and we summarize for the annealed Lyapunov exponent:
\[\beta(1) \geq \inf_{\{\Op\in \Me(\Omega)\}} \left[\E^{\Op}[F]+J(\Op) \right]\, .\]
Due to the positivity of $F$ we see that $\E^{\Op}[F]\geq 0$ for all $\Op \in \Me(\Omega)$. By the definition of $J$ and the finite expectation of $\nu$ the infimum will not be reached for a not shift invariant measure. We conclude
\begin{align*}
\beta(1) \geq  \inf_{\{\Op\in \Mt(\Omega)\}} \left[\E^{\Op}[F]+H(\Op|\Prob)\right] \, . 
\end{align*}
\end{proof}

\section{Lyapunov exponents on trees} \label{trees}
The aforegoing result could have been mainly proven due to the additivity of $a$. This additivity property does not hold for simple symmetric random walks on the higher dimensional lattices $\Z^d$ with $d \geq 2$. Could there nevertheless be other structures on which the random walk moves, where we gain again the additivity of $a$?  One possible answer is a $d$-regular tree.\\
In this section we will see how we can apply our main result to Lyapunov exponents on a $d$-regular tree. First of all we give a short summary of the main ingredients of a regular tree. For a more precise introduction in the context of random walks, see e.g. \cite{woess}. Second we will introduce Lyapunov exponents on an infinite tree.

\subsection{Random walk with random killing on infinite d-regular trees}\label{treesection}
Let $d\geq 2$ and let $T_d$ denote the $d$-regular infinite tree. We call $V(T_d)$ the set of vertices of the tree and $E(T_d)$ the set of edges. Then we define a symmetric nearest neighbour random walk $(Z_n)_{n\in \N}$ on $T_d$ by choosing a starting point
\[Z_0=v\]
for $v\in V(T_d)$ and the transition probabilities given by
\begin{align}
p(x,y)=\begin{cases} 
\nicefrac{1}{d} & \text{if} \,\ x \sim y\\
0 & \text{else} \, .
\end{cases}
\end{align}
The relation $x \sim y$ for $x,y\in V(T_d)$ means that these two vertices are neighbours, i.e. $[x,y]$ is an edge in $E(T_d)$. Thus $(Z_n)_{n\in \N}$ is a spatially homogeneous Markov chain adapted to $T_d$. We will denote its path measure by $\Pa_v^T$ and the corresponding expectation value by $\Er_v^T$. 

As we have done before on the integers we attach to each vertex $x\in V(T_d)$ a random potential $\omega(x)$ and assume that $\omega:=(\omega(x))_{x\in V(T_d)}$ is a family of nonnegative random variables which are i.i.d. by $\nu$ on $[0,\infty)$.  From this we obtain again the canonical probability space $(\Omega, \F, \Prob)$ described by
\[\Omega := [0,\infty)^{V(T_d)}\]
with its usual product $\sigma$-algebra $\F$
and the product measure
\begin{align*}
\Prob:=\bigotimes_{x\in V(T_d)}\nu \, .
\end{align*}
The expectation value with respect to $\Prob$ will be denoted by $\E^{\Prob}$. We assume again that 
$\nu$ has finite expectation.

As $T_d$ is a tree, it contains no cycles and we have for each $x,z \in V(T_d)$ an unique shortest path $\pi=[v,...,z]$ from $v$ to $z$. The distance between two vertices is defined as the length of this shortest path:
\[d(x,z):=\vert[x,...,z]\vert\, .\]
We call a sequence of distinct vertices $\pi=[\ldots, x_{-2},x_{-1},x_{0},x_{1},x_{2}, \ldots]\subset V(T_d)$ which satisfy $x_j \sim x_{j+1}$ for all $j\in \Z$ a \textit{geodesic}. If we have an one-sided infinite sequence $\pi=[x_{0},x_{1},x_{2}, \ldots]\subset V(T_d)$ with $x_j \sim x_{j+1}$ for all $j\geq 1$ we call it \textit{ray}. 
 
Now we are interested in the riskiness of walking around on this tree equipped with the random potentials. More precisely we want to observe how risky journeys along a fixed geodesic are. Here $\tau_{x},x \in V(T_d)$, is defined by
\[\tau_x:=\inf\{n\geq 0\, : \, Z_n=x\}\]
for each $x \in V(T_d)$. 

The two-point-functions defined in the following section are the counterparts to our well know functions $e,f,a$ and $b$ from the first part of the paper.
\begin{definition} Let $x,y \in V(T_d)$ be two vertices of the tree and $\omega$ a realization of the random potentials. We define:
\begin{align*}
F_q(x,y,\omega)&:=\Er_x^T \left[\exp \left(-\sum_{k=0}^{\tau_y-1} \omega (Z_k) \right), \tau_y< \infty\right] \, \\
A(x,y,\omega)&:=-\ln F_q(x,y,\omega)\, \\
F_a(x,y)&:=\E^{\Prob}\left[\Er_x^T \left[\exp \left(-\sum_{k=0}^{\tau_y-1} \omega (Z_k) \right), \tau_y< \infty\right]\right]\, \\
B(x,y)&:=-\ln F_a(x,y)\, .
\end{align*}
\end{definition}

The two functions $F_q$ and $F_a$ denote the probability that the random walk reaches $y$ after starting at $x$, for the quenched environment where the potentials are frozen and the averaged environment. In contrast to $e$ and $f$ the random walk here is driven by the different path measure $\Prob_x^T$ on the tree. 

Let us now fix a geodesic $\pi^*=[\ldots, x_{-2},x_{-1},x_{0},x_{1},x_{2}, \ldots]$ and look at the behaviour of $F_q(x_0,x_i,\omega)$ and $F_a(x_0,x_i)$ in the long run, that is if $d(x_0,x_i) \rightarrow \infty$. Having a closer look on the structure of the tree we see that each trajectory of the random walk from $x_i$ to $x_j$ contains the path $[x_i,x_j]\subset \pi$. That is, the random walk has to pass all these points on its journey at least once. Only the excursions to the branches beside the geodesic vary. For each $x,y \in V(T_d)$ a branch $T_{x,y}$ of the tree is defined by 
\[T_{x,y}:=\{v \in V(T_d): y \in \pi(x,v)\}\, .\]
This allows a modification of the model in the following way: We can combine all the risk, which the random walk has to face during its excursions into a modified potential for each point of the geodesic and in the end we identify the geodesic with the integers. For the latter let $x_0 \in V(T_d)$ be a fixed starting point of the random walk on the geodesic. Then we can split $\pi^*$ into two rays 
\begin{align*}
\pi^+=[x_{0},x_{1},x_{2}, \ldots] \,\ \text{and} \,\ \pi^-=[x_{0},x_{-1},x_{-2}, \ldots]\, .
\end{align*}
We identify the geodesic with the integers via 
\begin{align}
\lambda:\pi^* \rightarrow \Z, \, x_i\mapsto
\begin{cases}
 \mskip15mu d(x_i,x_0) & \text{if} \,\ x_i \in \pi^+\\
-d(x_i,x_0) & \text{if} \,\ x_i \in \pi^- \, \\
\end{cases} \label{integeridentification}
\end{align}
and vice versa. For modifying the potential we have to spend more effort. Moving on an infinite tree contains, besides the given potentials, the additional risk of leaving the geodesic and getting lost in the corresponding branch of the tree. Getting lost is understood as the random walk not returning to $\pi^*$ again in finite time and disappearing at infinity of the tree within such a branch.

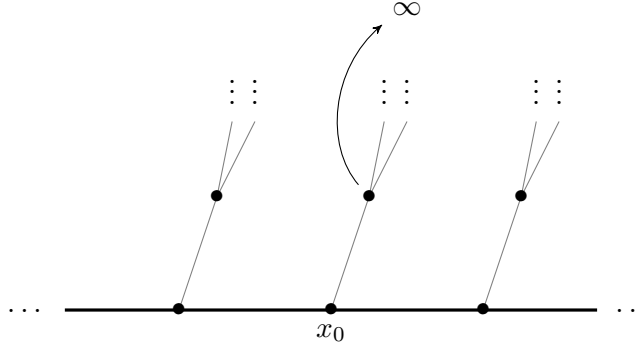
\begin{figure}[h]
\centering
\begin{tikzpicture}[x=1cm,y=1cm]
\draw[help lines,step=1, thin, draw=white] (-4,-1) grid (4,4);

\draw[thin,gray] (0,0) -- (0.5,1.5);
\draw[thin,gray] (2,0) -- (2.5,1.5);
\draw[thin,gray] (-2,0) -- (-1.5,1.5);

\draw[thin,gray] (0.5,1.5) -- (0.7,2.5);
\draw[thin,gray] (0.5,1.5) -- (1,2.5);

\draw[thin,gray] (2.5,1.5) -- (2.7,2.5);
\draw[thin,gray] (2.5,1.5) -- (3,2.5);

\draw[thin,gray] (-1.5,1.5) -- (-1.3,2.5);
\draw[thin,gray] (-1.5,1.5) -- (-1,2.5);

\foreach \Point in {(0,0),(2,0),(-2,0)}{
    \node at \Point {{\textbullet}};
    }

\foreach \Point in {(0.5,1.5),(2.5,1.5),(-1.5,1.5)}{
    \node at \Point {{\textbullet}};
    }

\foreach \Point in {(-4,0),(4,0)}{
    \node at \Point {{\ldots}};
    }

\draw[draw, very thick] (0,0) -- (3.5,0);
\draw[draw, very thick] (-3.5,0) -- (0,0);

\foreach \Point in {(0.7,3),(1,3),(2.7,3),(3,3),(-1.3,3),(-1,3)}{
    \node at \Point {{\vdots}};
    }
    
\node[] at (0,-0.3) {\textbf{$x_0$}};
\node[] at (1,4) {\textbf{$\infty$}}
 edge[pil,thin,<-, bend left=-45] node[auto] {} (0.5,1.5);

\end{tikzpicture}
\caption{A $3$-regular tree oriented on a fixed geodesic (black).}
\end{figure}

Consequently there are three possibilities how the random walk can die at a point $x_i$ of the geodesic: Firstly due to the given potential $\omega(x_i)$, secondly due to the potentials on the way on a finite excursion into a branch, thirdly because the random walk can disappear to infinity. We want to combine these three risks in a new sequence of potentials. To do so we define for $x_i \in \pi^*$ the new stopping time
\begin{align}
\sigma_{x_i}:=\inf\{n\geq 1: Z_n=x_{i+1}\,\ \text{or} \,\ Z_n=x_{i-1}\}\, . \label{stoppingtimetree}
\end{align} 
This is the first time the random walk on the tree hits the left or right neighbour of vertex $x_i$ on the geodesic. This stopping time is finite with a positive probability: Using the generating function technique described in \cite[Chapter 9]{woess}, it holds 
\begin{align*}
L(z):=& \sum_{n=1}^{\infty} \Pa_{x_i}^T[\sigma_{x_i}=n]\, z^n \\
=& \frac{2}{d} z+ \frac{d-2}{d} z \, F(z)\,  L(z) 
\end{align*}
where $F(z):=\sum_{k=1}^{\infty} \Pa_{x}^T[\tau_y =k]\, z^k$ for each $x,y\in V(T_d)$ with $y\sim x$. As the tree is regular and the random walk is space homogeneous, $F(z)$ 
is equal for each pair of neighbours. $F(1)$ is the probability that the random walk reaches a given neighbour of the starting point in finite time. It is well known, that $F(1)=\frac{1}{d-1}$ and we conclude that
\[0<\, \Pa_{x_i}^T[\sigma_{x_i} < \infty]=L(1)= \frac{2(d-1)}{(d-1)^2+1}\, \leq 1 \,  .\]  
Furthermore we define for each $x_i\in \pi^*$ the new random variable
\begin{align*}
h(x_i,\omega):=\Er_{x_i}^T\left[\exp\left(-\sum_{k=0}^{\sigma_{x_i}-1} \omega(Z_k)\right), \, \sigma_{x_i}<\infty\right] ,
\end{align*}
which is the expected probability that the random walk starting at $x_i$ survives its finite excursions to some branches before proceeding along the fixed geodesic $\pi^*$. Thus the probability to die at a vertex $x_i$, namely $1-h(x_i,\omega)$,
comprises the three types of possible risks outlined above. In order to use the aforegoing theory we perform a slight modification of $h$ to obtain the final new random potentials on the geodesic. 

\begin{proposition}\label{rhoprop}
Let $\pi^*\subset V(T_d)$ be a fixed geodesic. At each vertex $x_i\in \pi^*$ the random walk moving along $\pi^*$ survives with probability $e^{-\rho(x_i, \omega)}$
where $\rho=(\rho(x_i,\omega))_{x_i\in \pi^*}$ are positive i.i.d. random variables on $\Omega$ given by
\begin{align} \rho(x_i,\omega):=-\ln h(x_i,\omega)\, . \label{potentialtree}
\end{align}
Moreover, each $\rho(x_i,\omega)$ has finite expectation w.r.t. $\Prob$.
\end{proposition}
\begin{proof}
We have seen before that the random walk survives each site on its journey with probability $e^{-\rho(x_i, \omega)}$. 
As $h(x_i,\omega)$ is $(0,1]$-valued, $\rho$ is nonnegative. Additionally, $\omega \mapsto \rho(x_i,\omega)$ is for each $x_i\in \pi^*$ a continuous function with respect to the product topology on $\Omega$ and consequently measurable. Due to the definition of $\sigma_{x_i}$ the movement of the random walk within the event which defines $h(x_i,\omega)$ is restricted to the union of $\{x_i\}$ with the branches $T_{x_i,y}$ with $y\neq \{x_{i+1},x_{i-1}\}$. This, together with the identical distribution of $(\omega(i))_{i\in T_d}$, implies that $(\rho(x_i,\omega))_{x_i\in \pi}$ are i.i.d. by the image measure $\tilde{\nu}$ defined by
\begin{align} 
\tilde{\nu}[B]=\nu\left[\{\omega:\rho(x_i,\omega)\in B \}\right] \label{imagemeasuremu}
\end{align}
for all Borel-sets $B\in \mathcal{B}([0,\infty))$. 
For the last statement we observe that
$\big\{ \sigma_{x_i}=1\big\} \subseteq \big\{ \sigma_{x_i}<\infty\big\}$
which implies
\begin{align*}
\E^{\Prob}[\rho(x_i,\omega)]&=\E^{\Prob}\left[-\ln \Er_{x_i}^T\left[\exp\left(-\sum_{k=0}^{\sigma_{x_i}-1} \omega(Z_k)\right)\ind_{\{\sigma_{x_i} < \infty\}}\right] \right]\\
&\leq\E^{\Prob}\left[-\ln \Er_{x_i}^T\left[\exp\left(-\sum_{k=0}^{\sigma_{x_i}-1} \omega(Z_k)\right) \ind_{\{\sigma_{x_i}=1\}} \right] \right]\\[2mm]
&=\E^{\Prob}\left[-\ln \Er_{x_i}^T\left[\exp\left(-\omega(x_i)\right) \ind_{\{\sigma_{x_i}=1\}}\right] \right]\\[2mm]
&=\E^{\Prob}[\omega(x_i)]-\ln\left(\frac{2}{d}\right)\, .
\end{align*}
As the random potentials $\omega$ are supposed to have finite expectation, this holds for  $\rho$ as well.
\end{proof}
We obtain for the new potentials $\rho$ the slightly modified probability space $(\tilde{\Omega}:=[0,\infty)^{\Z}, \tilde{\F}, \tilde{\Prob}:=\otimes_{\Z}\tilde{\nu})$ where $\tilde{\F}$ is the usual borelian product-$\sigma$-algebra and $\tilde{\nu}$ defined as in (\ref{imagemeasuremu}). If  $\rho$ is given, we just have to know how often the random walk visits the different sites of the geodesic. As the random walk $(Z_n)_{n\in \N}$ is a symmetric random walk, it holds that
\[\Pa_{x_i}^T[Z_{\sigma_{x_i}}=x_{i+1}|\sigma_{x_i}<\infty]=\frac{1}{2}=\Pa_{x_i}^T[Z_{\sigma_{x_i}}=x_{i-1}|\sigma_{x_i}<\infty]\, . \]
Therefore conditionally upon the events $[\sigma_{x_i}< \infty]$ the sequence of random variables $(\tilde{S}_n)_{n\in \N}$ defined by
\begin{align*}
&\tilde{S}_0:=x_0 \,\ \text{and} \,\ \tilde{S}_n:=Z_{\sigma_{\tilde{S}_{n-1}}}
\end{align*}
is a simple symmetric random walk on the geodesic $\pi^*$ with starting point $x_0$. This, together with (\ref{integeridentification}) enables us to define the Lyapunov exponents along a geodesic in the well-known manner for the random walk $(\tilde{S}_n)_{n\in \N}$ and the random potentials $\rho(\omega)$. Always $(\ref{integeridentification})$ in mind, we will still denote the sites of the geodesic by its original name instead of integer numbers. 
Let $\tilde{a}$ and $\tilde{b}$ be defined as in Definitions \ref{quencheddef} and \ref{annealeddef} but for $(\tilde{S}_n)_{n\in \N}$ and $\rho$. Then the finite expectation of $\rho$ guarantees the existence of the Lyapunov exponents $\tilde{\alpha}$ and $\tilde{\beta}$ as in Propositions \ref{existencequenched} and \ref{existenceannealed}. Moreover we see the following identity where $\tilde{\tau}_x$ denotes the first point of time where the random walk $(\tilde{S}_n)_{n\geq 0}$ hits $x\in \pi^*$:
\begin{align*}
&\tilde{a}(x_0,x_i,\rho):=-\ln \Er_{x_0}^{\tilde{S}}\left[\exp\left(-\sum_{k=0}^{\tilde{\tau}_{x_i}-1} \rho(\tilde{S}_k)\right), \tilde{\tau}_{x_i}<\infty\right] \\[2mm]
&= - \ln \Er_{x_0}^{\tilde{S}}\left[\prod_{k=0}^{\tilde{\tau}_{x_i}-1} \Er_{\tilde{S}_k}^T\left[\exp\left(-\sum_{m=0}^{\sigma_{\tilde{S}_k}-1} \omega(Z_m)\right)  \ind_{\sigma_{\tilde{S}_k}<\infty\}} \right] \ind_{\{\tilde{\tau}_{x_i}<\infty\}}\right]\\[2mm]
&= - \ln \Er_{x_0}^{\tilde{S}}\left[ \Er_{x_0}^T\left[\exp\left(-\sum_{m=0}^{\sigma_{\tilde{S}_{\tilde{\tau}_{x_i}-1}}-1} \omega(Z_m)\right) \vphantom{\ind_{\{\sigma_{\tilde{S}_{\tilde{\tau}_{x_i}-1}}<\infty\}}} \ind_{\{\sigma_{\tilde{S}_{\tilde{\tau}_{x_i}-1}}<\infty\}} \right]  \ind_{\tilde{\tau}_{x_i}<\infty\}}\right] \\[2mm]
&= - \ln \Er_{x_0}^{\tilde{S}}\left[ \Er_{x_0}^T\left[\exp\left(-\sum_{m=0}^{\sigma_{x_{i-1}}-1} \omega(Z_m)\right) \ind_{\{\sigma_{x_{i-1}}<\infty\}} \ind_{\{Z_{\sigma_{x_{i-1}}}=x_i\}} \right] \right.\\
&\mskip80mu +\Er_{x_0}^T\left[\exp\left(-\sum_{m=0}^{\sigma_{x_{i+1}}-1} \omega(Z_m) \ind_{\{\sigma_{x_{i+1}}<\infty\}} \ind_{\{Z_{\sigma_{x_{i+1}}}=x_i\}} \right]
 \ind_{\tilde{\tau}_{x_i}<\infty\}}\right]\\[2mm]
&= - \ln \Er_{x_0}^{\tilde{S}}\left[ \Er_{x_0}^T\left[\exp\left(-\sum_{m=0}^{\tau_{x_i}-1} \omega(Z_m)\right) \ind_{\{\tau_{x_i}<\infty\}} \right] \ind_{\{\tilde{\tau}_{x_i}<\infty\}}\right]\\[2mm]
&= - \ln \Er_{x_0}^T\left[\exp\left(-\sum_{m=0}^{\tau_{x_i}-1} \omega(Z_m)\right), \tau_{x_i}<\infty \right]\\[2mm]
&=A(x_0,x_i,\omega)
\end{align*}
for all $x_i \in \pi^*$. The same holds for $\tilde{b}$ and $B$. 
This provides the existence of the Lyapunov exponents along $\pi^*$: Let $x_i \in \pi^*$, then 
\begin{align*}
\tilde{\alpha}(x_i)&=\lim_{n\rightarrow \infty} \frac{1}{n}{A}(x_0,x_{ni},\omega) = \lim_{n\rightarrow \infty} \frac{1}{n}\E[A(x_0,x_{ni},\omega)]\\
&=\inf_{n\in N} \frac{1}{n}\E[A(x_0,x_{ni},\rho)] \quad \text{} 
\end{align*}
exists ${\tilde{\Prob}}$-a.s. and $L^1(\tilde{\Prob})$ and furthermore we have
\begin{align*}
\tilde{\beta}(x_i)&= \lim_{n\rightarrow \infty} \frac{1}{n}{B}(x_0,x_{ni})=\inf_{n\in N} \frac{1}{n} {B}(x_0,x_{ni}) \, .
\end{align*}
Due to the properties of $\rho$ in Proposition \ref{rhoprop} we can apply Theorem \ref{theorem} as well and see that our proven relation does hold for Lyapunov exponents on $d$-regular trees:
\[\tilde{\beta}(x_1)=\inf_{\Op}\{\E\left[ A(x_0,x_1, \rho)\right] +H(\Op|\tilde{\Prob})\}\, ,\]
where the infimum runs over all shift invariant probability measures on $\tilde{\Omega}$.

\subsection{Non-symmetric random walk on trees}
Let us go back to the integers for a moment. From now on we consider a nearest neighbour walk which is not symmetric. That is in each step it jumps independently of all the steps before with probability $p$ to the right and with probability $1-p$ to the left. It is easy to see that Propostion \ref{existencequenched}, Proposition \ref{existenceannealed} and Theorem \ref{relation} hold for this non-symmetric nearest neighbour random walk as well. 

There is also a non-symmetric random walk counterpart on infinite regular trees. To define this we fix a root $o \in V(T_d)$ and a geodesic $\Theta$ and consider a particular representation of infinite regular trees (see Figure \ref{fig:label3}). This representation displays the generations $H_i$, $i \in \Z$, of a tree. We write 
\[ \h(x):= i \,\ \mbox{if} \,\ x \in H_i \]
for each vertex $x \in V(T_d)$. For two vertices $x,y \in V(T_d)$ we call $y$ a \textit{predecessor} of $x$ if $\h(y)=\h(x)-1$ and a \textit{child} if  $\h(y)=\h(x)+1$. Note that each $H_i$ is infinite.

\begin{figure}[h] \label{generationtree}
\centering
\begin{tikzpicture}[x=0.8cm,y=0.8cm]
\draw[help lines,step=1, thin, draw=white] (-3,-4) grid (5,4);

\draw[thin] (3.6,9) -- (-1.8,-4.5);
\draw[thin] (2.2,5.5) -- (4,4.5);
\draw[thin] (0,0) -- (0.8,-2);
\draw[thin] (0.8,-2) -- (0.4,-3.5);
\draw[thin] (0.8,-2) -- (1.2,-3.5);
\draw[thin] (-0.8,-2) -- (-0.4,-3.5);
\draw[thin] (1.2,-3.5) -- (1.4,-4.5);
\draw[thin] (1.2,-3.5) -- (1,-4.5);
\draw[thin] (0.4,-3.5) -- (0.2,-4.5);
\draw[thin] (0.4,-3.5) -- (0.6,-4.5);
\draw[thin] (-0.4,-3.5) -- (-0.6,-4.5);
\draw[thin] (-0.4,-3.5) -- (-0.2,-4.5);
\draw[thin] (-1.4,-3.5) -- (-1,-4.5);
\draw[thin] (1,2.5) -- (3.2,0);
\draw[thin] (3.2,0) -- (2.4,-2);
\draw[thin] (3.2,0) -- (4,-2);
\draw[thin] (2.4,-2) -- (2.0,-3.5);
\draw[thin] (2.4,-2) -- (2.8,-3.5);
\draw[thin] (4,-2) -- (3.6,-3.5);
\draw[thin] (4,-2) -- (4.4,-3.5);
\draw[thin] (2.0,-3.5) -- (1.8,-4.5);
\draw[thin] (2.0,-3.5) -- (2.2,-4.5);
\draw[thin] (2.8,-3.5) -- (2.6,-4.5);
\draw[thin] (2.8,-3.5) -- (3,-4.5);
\draw[thin] (3.6,-3.5) -- (3.4,-4.5);
\draw[thin] (3.6,-3.5) -- (3.8,-4.5);
\draw[thin] (4.4,-3.5) -- (4.2,-4.5);
\draw[thin] (4.4,-3.5) -- (4.6,-4.5);

\draw[dashed,gray,thin] (-2.3,0) -- (5.5,0);
\draw[dashed,thin,gray] (-2.3,2.5) -- (5.5,2.5);
\draw[dashed,thin,gray] (-2.3,5.5) -- (5.5,5.5);

\node[] at (-2.9,0) {\textbf{$H_{1}$}};
\node[] at (-2.9,2.5) {\textbf{$H_{0}$}};
\node[] at (-2.9,5.5) {\textbf{$H_{-1}$}};

\node[] at (0.75,2.65) {\textbf{$o$}};


\node[] at (3.3,8.9) {\textcolor{black}{{$\Theta$}}};

\node[] at (1,2.5) {}
 edge[pil,thin,->, bend left=30,thick] node[auto] {\small{$p$}} (2.2,5.5)
 edge[pil,thin,->, bend left=30,thick] node[auto] {\textbf{$\frac{1-p}{2}$}} (3.2,0);

\node[] at (0,0) {}
 edge[pil,thin,<-, bend left=30,thick,below] node[auto] {$\frac{1-p}{2}$} (1,2.5);

\foreach \Point in {
(2.2,5.5),
(1,2.5), 
(0,0),(3.2,0),
(-0.8,-2),(0.8,-2),(2.4,-2),(4,-2),
(-1.4,-3.5),(-0.4,-3.5),(0.4,-3.5),(1.2,-3.5),(2.0,-3.5),(2.8,-3.5),(3.6,-3.5),(4.4,-3.5),
(-1.8,-4.5),(-1,-4.5),(-0.6,-4.5),(-0.2,-4.5),(0.2,-4.5),(0.6,-4 .5),(1,-4.5),(1.4,-4.5),(1.8,-4.5),(2.2,-4.5),(2.6,-4.5),(3.0,-4.5),(3.4,-4.5),(3.8,-4.5),(3.8,-4.5),(4.2,-4.5),(4.6,-4.5)}{
    \node[] at \Point {{\textbullet}};
    }

\end{tikzpicture}
\caption{A $3$-regular infinite tree.}
    \label{fig:label3}
\end{figure}
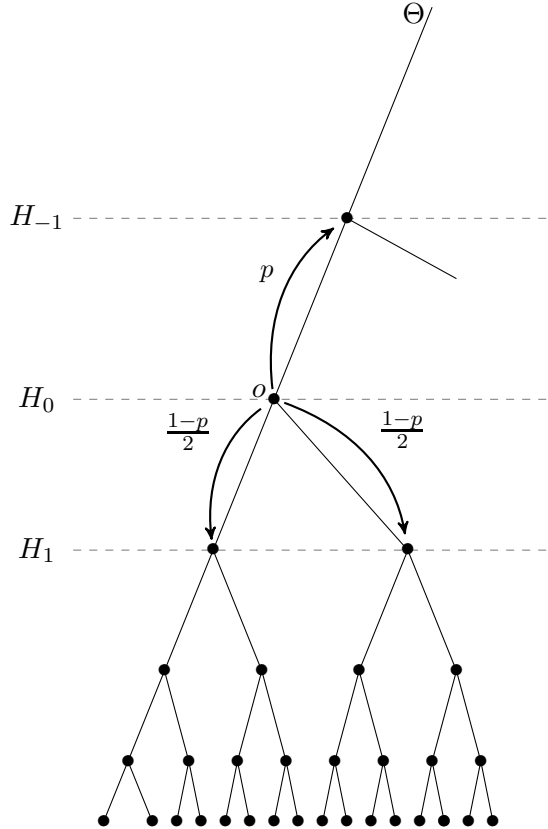
Then we define the random walk with drift $(Z_n)_{n\in \N}$ in the following way: The random walk jumps from its current position to the predecessor with probability $p$ and to a child with probability ${1-p}$. This yields the following transition probabilities for $u,v \in V(T_d)$
\begin{align}
\Pa^p[Z_{n+1}=v\mid Z_n=u]=
\begin{cases}
p & \mbox{if} \,\ u\sim v\,\ \text{and}\,\ \h(v)=h(u)-1 \\
\frac{1-p}{d-1} & \mbox{if}\,\  u\sim v\,\ \text{and}\,\ \h(v)=h(u)+1  \\
0 & \mbox{else} \,  \\
\end{cases} \label{notsymmetrictransition}
\end{align}
for each $n \in \N$. The second line is due to the fact that one vertex has $d-1$ children in the next generation. 

There are two kinds of infinite geodesics: the first has two ends downwards (like the thick geodesic in Figure \ref{fig:label1}), the second has one vertex in each generation (like the thick geodesic in Figure \ref{fig:label2}). 

Firstly we take a geodesic $\pi=[\ldots,x_{-1},x_0,x_1,\ldots]$ of the latter form, that is  $\h(x_{i+1})=\h(x_i)+1$ for each $i\in \Z$. 
Then it holds that
\[\Pa_{x_i}^p[Z_{\sigma_{x_i}}=x_{i+1}|\sigma_{x_i}<\infty]=\frac{p}{p+\frac{1-p}{d-1}}=1-\Pa_{x_i}^p[Z_{\sigma_{x_i}}=x_{i-1}|\sigma_{x_i}<\infty]\,  \]
where $\Pa^p_{x_i}$ is the path measure of the non-symmetric random walk induced by (\ref{notsymmetrictransition}) with starting point $x_i$, and the stopping times $\sigma_{x_i}$ are defined as in (\ref{stoppingtimetree}) with respect to $\Pa^p$. Consequently the new sequence of random variables $(\tilde{S}_n)_{n\in \N}$ defined by $\tilde{S}_0:=x_0$ and $\tilde{S}_n:=Z_{\sigma_{\tilde{S}_{n-1}}}$ is a non-symmetric random walk on $\pi$. This together with Proposition \ref{rhoprop} provides again the existence of the quenched and annealed Lyapunov exponents and also their relation using relative entropy. 

Let us consider now a geodesic $\pi=[\ldots,x_{-1},x_0,x_1,\ldots]$ whose two ends point downwards. Such a geodesic has an highest point $x_k\in\pi$ like in Figure \ref{fig:label1} and it holds
\begin{align*}
\h(x_{i+1})=
\begin{cases}
 \h(x_i)+1 \,\ \mbox{if}  \,\ i\geq k \\
\h(x_i)-1 \,\ \mbox{if}  \,\ i<k \, .
\end{cases}
\end{align*}
Looking at the transition probabilities of the random walk at $x_k$ we see   
\[\Pa^p[Z_{i+1}=x_{k+1}\mid Z_i=x_k]=\frac{1-p}{d-1}=\Pa^p[Z_{i+1}=x_{k-1}\mid Z_i=x_k]\]
respectively
\[\Pa^p[Z_{i+1}=x_k\mid Z_i=x_{k+1}]=p=\Pa^p[Z_{i+1}=x_k\mid Z_i=x_{k-1}]\, \]
for the two different neighbours of $x_k$ (see Figure \ref{fig:label1}).
At this vertex the direction of the drift changes and the immediate identification with the integers does not work. We need a slight modification of the aforegoing setting. Let the starting point of the random walk be $x_0$. We observe the travelling risk in the direction of $x_i\in \pi$ for $i>0$. 
\begin{figure}[h]
\begin{minipage}[t]{0.5\textwidth}
\begin{tikzpicture}[x=0.6cm,y=0.6cm]
\draw[help lines,step=1, thin, draw=white] (-3,-4) grid (5,4);

\draw[thin,gray] (3.6,9) -- (-1.8,-4.5);
\draw[thin,gray] (2.2,5.5) -- (4,4.5);
\draw[thin,gray] (0,0) -- (0.8,-2);
\draw[thin,gray] (0.8,-2) -- (0.4,-3.5);
\draw[thin,gray] (0.8,-2) -- (1.2,-3.5);
\draw[thin,gray] (-0.8,-2) -- (-0.4,-3.5);
\draw[thin,gray] (1.2,-3.5) -- (1.4,-4.5);
\draw[thin,gray] (1.2,-3.5) -- (1,-4.5);
\draw[thin,gray] (0.4,-3.5) -- (0.2,-4.5);
\draw[thin,gray] (0.4,-3.5) -- (0.6,-4.5);
\draw[thin,gray] (-0.4,-3.5) -- (-0.6,-4.5);
\draw[thin,gray] (-0.4,-3.5) -- (-0.2,-4.5);
\draw[thin,gray] (-1.4,-3.5) -- (-1,-4.5);
\draw[thin,gray] (1,2.5) -- (3.2,0);
\draw[very thick] (3.2,0) -- (2.4,-2);
\draw[very thick] (3.2,0) -- (4,-2);
\draw[very thick] (2.4,-2) -- (2.0,-3.5);
\draw[thin,gray] (2.4,-2) -- (2.8,-3.5);
\draw[thin,gray] (4,-2) -- (3.6,-3.5);
\draw[very thick] (4,-2) -- (4.4,-3.5);
\draw[thin,gray] (2.0,-3.5) -- (1.8,-4.5);
\draw[very thick] (2.0,-3.5) -- (2.2,-4.5);
\draw[thin,gray] (2.8,-3.5) -- (2.6,-4.5);
\draw[thin,gray] (2.8,-3.5) -- (3,-4.5);
\draw[thin,gray] (3.6,-3.5) -- (3.4,-4.5);
\draw[thin,gray] (3.6,-3.5) -- (3.8,-4.5);
\draw[thin,gray] (4.4,-3.5) -- (4.2,-4.5);
\draw[very thick] (4.4,-3.5) -- (4.6,-4.5);

\draw[dashed,gray,thin] (-2.4,0) -- (5.5,0);
\draw[dashed,thin,gray] (-2.4,2.5) -- (5.5,2.5);
\draw[dashed,thin,gray] (-2.4,5.5) -- (5.5,5.5);

\node[] at (-3,0) {\textbf{$H_{1}$}};
\node[] at (-3,2.5) {\textbf{$H_{0}$}};
\node[] at (-3,5.5) {\textbf{$H_{-1}$}};

\node[] at (0.7,2.8) {\textbf{$o$}};

\node[] at (3.6,0.3) {\textbf{{$x_k$}}};
\node[] at (4.9,-3.4) {\textbf{{$x_0$}}};
\node[] at (5.1,-4.5) {\textbf{{$x_i$}}};200


\node[] at (3.2,8.9) {\textcolor{black}{{$\Theta$}}};

\node[] at (3.2,0) {}
 edge[pil,thin, <-, bend left=30,thick] node[auto] {\small{$p$}} (4,-2);
 
\node[] at (2.4,-2) {}
 edge[pil,thin,->, bend left=30,thick,pos=0.4] node[auto] {\small{$p$}} (3.2,0);
 
 \node[] at (4.9,-3.4) {}
 ;

\foreach \Point in {
(2.2,5.5),
(1,2.5), 
(0,0),(3.2,0),
(-0.8,-2),(0.8,-2),(2.4,-2),(4,-2),
(-1.4,-3.5),(-0.4,-3.5),(0.4,-3.5),(1.2,-3.5),(2.0,-3.5),(2.8,-3.5),(3.6,-3.5),(4.4,-3.5),
(-1.8,-4.5),(-1,-4.5),(-0.6,-4.5),(-0.2,-4.5),(0.2,-4.5),(0.6,-4 .5),(1,-4.5),(1.4,-4.5),(1.8,-4.5),(2.2,-4.5),(2.6,-4.5),(3.0,-4.5),(3.4,-4.5),(3.8,-4.5),(3.8,-4.5),(4.2,-4.5),(4.6,-4.5)}{
    \node at \Point {{\textbullet}};
    }

\end{tikzpicture}

\caption{Original geodesic $\pi$}
    \label{fig:label1}
\end{minipage}
\begin{minipage}[t]{0.5\textwidth}
\begin{tikzpicture}[x=0.6cm,y=0.6cm]
\draw[help lines,step=1, thin, draw=white] (-3,-4) grid (5,4);

\draw[thin,gray] (1,2.5) -- (-1.8,-4.5);
\draw[very thick] (3.6,9) -- (1,2.5);
\draw[thin] (2.2,5.5) -- (4,4.5);
\draw[thin,gray] (0,0) -- (0.8,-2);
\draw[thin,gray] (0.8,-2) -- (0.4,-3.5);
\draw[thin,gray] (0.8,-2) -- (1.2,-3.5);
\draw[thin,gray] (-0.8,-2) -- (-0.4,-3.5);
\draw[thin,gray] (1.2,-3.5) -- (1.4,-4.5);
\draw[thin,gray] (1.2,-3.5) -- (1,-4.5);
\draw[thin,gray] (0.4,-3.5) -- (0.2,-4.5);
\draw[thin,gray] (0.4,-3.5) -- (0.6,-4.5);
\draw[thin,gray] (-0.4,-3.5) -- (-0.6,-4.5);
\draw[thin,gray] (-0.4,-3.5) -- (-0.2,-4.5);
\draw[thin,gray] (-1.4,-3.5) -- (-1,-4.5);
\draw[very thick] (1,2.5) -- (3.2,0);
\draw[thin,gray] (3.2,0) -- (2.4,-2);
\draw[very thick] (3.2,0) -- (4,-2);
\draw[thin,gray] (2.4,-2) -- (2.0,-3.5);
\draw[thin,gray] (2.4,-2) -- (2.8,-3.5);
\draw[thin,gray] (4,-2) -- (3.6,-3.5);
\draw[very thick] (4,-2) -- (4.4,-3.5);
\draw[thin,gray] (2.0,-3.5) -- (1.8,-4.5);
\draw[thin,gray] (2.0,-3.5) -- (2.2,-4.5);
\draw[thin,gray] (2.8,-3.5) -- (2.6,-4.5);
\draw[thin,gray] (2.8,-3.5) -- (3,-4.5);
\draw[thin,gray] (3.6,-3.5) -- (3.4,-4.5);
\draw[thin,gray] (3.6,-3.5) -- (3.8,-4.5);
\draw[thin,gray] (4.4,-3.5) -- (4.2,-4.5);
\draw[very thick] (4.4,-3.5) -- (4.6,-4.5);

\draw[dashed,gray,thin] (-2.4,0) -- (5.5,0);
\draw[dashed,thin,gray] (-2.4,2.5) -- (5.5,2.5);
\draw[dashed,thin,gray] (-2.4,5.5) -- (5.5,5.5);

\node[] at (-3,0) {\textbf{$H_{1}$}};
\node[] at (-3,2.5) {\textbf{$H_{0}$}};
\node[] at (-3,5.5) {\textbf{$H_{-1}$}};

\node[] at (0.7,2.8) {\textbf{$o$}};
\node[] at (2,2.8) {\tiny{$(\tilde{x}_{k-1})$}};

\node[] at (3.6,0.3) {\textbf{{$x_k$}}};
\node[] at (4.9,-3.4) {\textbf{{$x_0$}}};
\node[] at (5.1,-4.5) {\textbf{{$x_i$}}};

\node[] at (3.2,8.9) {\textcolor{black}{{$\Theta$}}};

\node[] at (3.2,0) {}
 edge[pil,thin, <-, bend left=30,thick] node[auto] {\small{$p$}} (4,-2);
 
\node[] at (1,2.5) {}
 edge[pil,thin,<-, bend left=30,thick,pos=0.5] node[auto] {\small{$p$}} (3.2,0);

\foreach \Point in {
(2.2,5.5),
(1,2.5), 
(0,0),(3.2,0),
(-0.8,-2),(0.8,-2),(2.4,-2),(4,-2),
(-1.4,-3.5),(-0.4,-3.5),(0.4,-3.5),(1.2,-3.5),(2.0,-3.5),(2.8,-3.5),(3.6,-3.5),(4.4,-3.5),
(-1.8,-4.5),(-1,-4.5),(-0.6,-4.5),(-0.2,-4.5),(0.2,-4.5),(0.6,-4 .5),(1,-4.5),(1.4,-4.5),(1.8,-4.5),(2.2,-4.5),(2.6,-4.5),(3.0,-4.5),(3.4,-4.5),(3.8,-4.5),(3.8,-4.5),(4.2,-4.5),(4.6,-4.5)}{
    \node at \Point {{\textbullet}};
    }

\end{tikzpicture}

\caption{Modified geodesic $\pi'$}
    \label{fig:label2}
\end{minipage}
\end{figure}

If $k\leq 0$ the vertex $x_k$ is equal to $x_0$ or not on the path $[x_0,x_i]$ for any $i>0$ and we perform the following modification: Firstly we cut the geodesic at $x_k$ and take the ray which includes $x_i$. Then we add to the ray the predecessor of $x_k$, which we call $\tilde{x}_{k-1}$. Adding the predecessor of $\tilde{x}_{k-1}$ and successively all the next predecessors, we obtain the modified geodesic $\pi'=[\ldots,\tilde{x}_{k-2},\tilde{x}_{k-1},x_k,x_{k+1},\ldots]$ like in Figure \ref{fig:label2}. This modification doesn't change $A(x_0,x_{ni},\omega)$ and $B(x_0,x_{ni})$ for all $n\in \N$ but the new geodesic $\pi'$ can be identified with the integers as before. 

If $k>0$ the point $x_k$ is on the path $[x_0,x_{mi}]$ for a $m\in\N$. But we know that it holds for each $n\geq m$  
\begin{align*}
F_q(x_0,x_{ni},\omega)&=F_q(x_0,x_k,\omega)\, F_q(x_k,x_{ni},\omega) \\
&=: C(\omega)  \cdot \, F_q(x_k,x_{ni},\omega) \\
F_a(x_0,x_{ni})&\geq \E[C(\omega)]\cdot \, F_a(x_k,x_{ni})\, 
\end{align*}
using the additivity property (Lemma \ref{additivitylemma}) and the FKG-inequality.
Consequently, we see 
\begin{align}
A(x_0,x_{ni},\omega)&=-\ln C(\omega ) + A(x_k,x_{ni},\omega) \label{othersideone}\\
B(x_0,x_{ni})&\leq -\ln \E[C(\omega)] + B(x_k,x_{ni})\, .\label{othersidetwo}
\end{align}
But because it is more probable to survive shorter journeys, it holds
\begin{align}
B(x_0,x_{ni})&\geq B(x_k,x_{ni}) \label{oneside}
\end{align}
as well. In $A(x_k,x_{ni},\omega)$ and  $B(x_k,x_{ni})$ the random walk starts at $x_k$ and travels to the direction of $x_{ni}$. Hence we are in the aforegoing case of an uniform drift on the direct path from $x_k$ to $x_{ni}$ and we can identify the geodesic with the integers by
cutting it and adding the predecessors. Thus, the Lyapunov exponents with starting point $x_k$
\begin{align*}
\tilde{\alpha}(x_i)&:=\lim_{n\rightarrow\infty} \frac{1}{n} A(x_k,x_{ni},\omega) \\
\tilde{\beta}(x_i)&:=\lim_{n\rightarrow\infty} \frac{1}{n} B(x_k,x_{ni}) 
\end{align*}
exist due to the same arguments as in Section \ref{treesection}. This, together with (\ref{othersideone}),(\ref{othersidetwo}) and (\ref{oneside}) implies
\begin{align*}
\tilde{\alpha}(x_i)&=\lim_{n\rightarrow\infty} \frac{1}{n} A(x_0,x_{ni},\omega) \\
\tilde{\beta}(x_i)&=\lim_{n\rightarrow\infty} \frac{1}{n} B(x_0,x_{ni}) 
\end{align*}
and by coincidence of the limits it also holds that
\[\tilde{\beta}(x_1)=\inf_{\Op}\{\E\left[ A(x_0,x_1, \rho)\right] +H(\Op|\tilde{\Prob})\}\, ,\]
where the infimum runs over all shift invariant probability measures on $\tilde{\Omega}$.\\ \\
\small{\textbf{Acknowledgement:} I would like to extend my sincere gratitude to Prof. Martin Zerner for proposing the study of Lyapunov exponents to me and many helpful discussions and ideas. I thank PD Elmar Teufl who raised the idea of working on trees and Prof. Wolfgang Woess for supporting my work.}
\bibliographystyle{alpha}

\bibliography{literatur}

\end{document}